\documentclass[12pt]{amsart}

\usepackage{amsmath,amssymb}

\theoremstyle{plain}
\newtheorem{thm}{Theorem}[section]
\newtheorem{theorem}[thm]{Theorem}

\newtheorem{lemma}[thm]{Lemma}
\newtheorem{corollary}[thm]{Corollary}
\newtheorem{proposition}[thm]{Proposition}
\theoremstyle{definition}
\newtheorem{remark}[thm]{Remark}

\newtheorem{notation}[thm]{Notation}
\newtheorem{definition}[thm]{Definition}

\newtheorem{example}[thm]{Example}

\newtheorem{question}[thm]{Question}

\numberwithin{equation}{section}

\newcommand{\sC}{{\mathcal C}}

\newcommand{\sF}{{\mathcal F}}

\newcommand{\sI}{{\mathcal I}}

\newcommand{\sK}{{\mathcal K}}

\newcommand{\sO}{{\mathcal O}}
\newcommand{\sP}{{\mathcal P}}

\newcommand{\sT}{{\mathcal T}}

\newcommand{\sV}{{\mathcal V}}

\newcommand{\C}{{\mathbb C}}

\newcommand{\BP}{{\mathbb P}}

\newcommand{\BS}{{\mathbb S}}

\newcommand{\fg}{{\mathfrak g}}

\newcommand{\fp}{{\mathfrak p}}

\newcommand{\aut}{{\mathfrak a}{\mathfrak u}{\mathfrak t}}

\def\Hom{\mathop{\rm Hom}\nolimits}

\title[Isotrivial VMRT-structures]{Isotrivial VMRT-structures of complete intersection type}
\author{Baohua Fu and Jun-Muk Hwang}
\thanks{Baohua Fu is supported by National Natural Science Foundation of China (11321101 and 11688101). He would like to thank KIAS and Max Planck Institute at Bonn for the hospitality. Jun-Muk Hwang is
supported by National Researcher Program 2010-0020413 of NRF}

\begin{document}
\maketitle

\begin{center}
{\em Dedicated to Ngaiming Mok on his sixtieth birthday}
\end{center}

\begin{abstract}
The family of varieties of minimal rational tangents on a quasi-homogeneous projective manifold is isotrivial. Conversely, are projective manifolds with isotrivial varieties of minimal rational tangents quasi-homogenous?  We will show that this is not true in general, even when the projective manifold has Picard number 1. In fact, an isotrivial family of varieties of minimal rational tangents needs not be locally flat in differential geometric sense. This leads to the question for which projective variety $Z$, the $Z$-isotriviality of varieties of minimal rational tangents implies local flatness. Our main result verifies this for many cases of $Z$ among complete intersections.
\end{abstract}

\section{Introduction}
We work in the complex analytic category. Recall that a projective variety is quasi-homogeneous  if the action of its automorphism group has a dense open orbit.
It is an intriguing question {\em how to recognize quasi-homogeneous varieties when the varieties are  a priori defined without apparent relations with any group actions}.
For uniruled projective manifolds,   the method of VMRT-structures introduced by Ngaiming Mok and the second author  (see \cite{Hw01} and \cite{HM99} for introductory surveys) can be applied to study this question.    Recall that given a uniruled projective manifold $X$, if we pick a family $\sK$ of minimal rational curves on $X$, we have its VMRT-structure  $\sC \subset \BP T(X)$, the subvariety defined as the closure of the union of tangent directions of members of $\sK$ through general points of $X$.  How do we use this to prove a variety is homogeneous or quasi-homogeneous? Quite often, the essential point lies in proving the VMRT-structure is locally homogeneous in the following sense.

\begin{definition}\label{d.vmrt}
 Let $X$ be a uniruled projective manifold and let $\sC \subset \BP T(X)$ be a VMRT-structure determined by a family $\sK$ of minimal rational curves.  Given a holomorphic vector field $v$ on an open subset $U\subset X$, we have the induced vector field $v^{\sharp}$ on $\BP T(U)$ obtained by the action of the local 1-parameter family of biholomorphisms generated by $v$. We say that $v$ {\em is a } $\sC$-{\em preserving vector field}, if $v^{\sharp}$ is tangent to $\sC|_U \subset \BP T(U)$.
 The VMRT-structure $\sC \subset \BP T(M)$ is said to be \begin{itemize} \item[(1)] {\em locally homogeneous} if $\sC$-preserving vector fields on some open subset $U \subset X$
 span the tangent space $T_x(U)$ at some point $x \in U$; and \item[(2)] {\em locally flat} if there are holomorphic coordinates $(x^1, \ldots, x^n)$ on a non-empty open subset $U \subset X$ such that the coordinate vector fields
  $\frac{\partial}{\partial x^1}, \ldots,  \frac{\partial}{\partial x^n}$ are $\sC$-preserving. \end{itemize} \end{definition}

These  local differential geometric properties of $\sC$  can be used to check that  certain varieties are  quasi-homogeneous, by the following results from Main Theorem in \cite{HM01} and Proposition 6.13 in \cite{FH}.

\begin{theorem}\label{t.CF}
Let $X$ be a Fano manifold of Picard number 1 and let $\sC \subset \BP T(X)$ be a VMRT-structure such that a general fiber $\sC_x \subset \BP T_x(X)$ is nonsingular and irreducible. If the VMRT-structure is locally homogeneous, then $X$ is quasi-homogeneous.
If the VMRT-structure is locally flat, then $X$ is an equivariant compactification of
the vector group $\C^n, n= \dim X$. \end{theorem}

This naturally leads to the question:
how do we check that a VMRT-structure is locally homogeneous?  Roughly speaking, this question is divided into two parts: the isotriviality and the vanishing of curvature.  The first part is to check that the natural projection $ \sC \to X$ is isotrivial as a family of projective varieties. More precisely, we will use the following terminology.

\begin{definition}\label{d.isotrivial}
Let $Z \subset \BP V$ be a submanifold of the projective space $\BP V$ where $V$ is a complex vector space with
$\dim V = \dim X$. A VMRT-structure $\sC \subset \BP T(X)$ is {\em isotrivial of type} $Z \subset \BP V$ (or equivalently, $Z$-{\em isotrivial}), if the fiber
$\sC_x \subset \BP T_x(X)$ over a general point $x \in X$  is isomorphic to $Z \subset \BP V$ by a linear isomorphism $T_x(X) \cong V$.
\end{definition}

The isotriviality is an obvious necessary condition for $\sC \subset \BP T(X)$ to be locally homogeneous.
Once the isotriviality is established, we can view $\sC \subset \BP T(X)$ restricted to a suitable Euclidean open subset $U \subset X$ as a kind of Cartan geometry and the question is reduced
to check that this Cartan geometry has vanishing curvature. In this paper, we will concentrate on this second part of the problem, where the main question is the following.  Suppose that we have chosen a nonsingular projective variety
$Z \subset \BP V$ which is  nondegenerate (i.e. does not lie on a hyperplane in $\BP V$).

\begin{question}\label{q.Question}
Let $\sC \subset \BP T(X)$ be a VMRT-structure on a uniruled projective manifold $X$. Assume that $\sC$ is isotrivial of type $Z \subset \BP V$.
Is $\sC$ locally homogenous or locally flat? \end{question}

The answer depends on $Z$. Affirmative answers are known in the following cases.
\begin{example}[\cite{M08}] \label{ex1} When $Z \subset \BP V$ is a homogeneous variety arising as
the VMRT of an irreducible Hermitian symmetric space, a $Z$-isotrivial VMRT-structure is locally flat.
\end{example} \begin{example}[\cite{Hw10}]\label{ex2} When $Z \subset \BP V$ is a nonsingular hypersurface of degree $\geq 4$, a $Z$-isotrivial VMRT-structure is locally flat.
\end{example}

We will give in Section \ref{s.example} some examples with negative answer to Question \ref{q.Question}, where $Z$ is given by linear sections of homogeneous varieties. In these examples, the variety $Z$ has continuous automorphisms. When $Z$ has no continuous automorphism,  no negative example to Question \ref{q.Question} is known, and in this case, a method to study Question \ref{q.Question} has been developed in \cite{Hw10} via Cartan's coframe formalism (this will be reviewed in
Section \ref{s.coframe}). The result of Example \ref{ex2} has been obtained by this approach.
The goal of this paper is to push this method further to cover some other smooth complete intersections. Our main result is the following.

\begin{theorem} \label{t.Main}
Let $Z \subset \BP V$ be a smooth non-degenerate complete intersection. Let us denote its multi-degree
by $[m_1, \ldots, m_c],$ where $ c$ is the codimension of $Z.$
Assume further that $Z$ satisfies one of the following:

(i) $Z$ is a  curve  of multi-degree different from $[3]$, $[4]$, $[2,2]$, $[2,3]$ or $[2,2,2]$;

(ii) $Z \subset \BP V$ is covered by lines, with multi-degree different from $[2]$, $[3]$, or $[2,2]$;

(iii) $Z$ is contained in a smooth hypersurface of degree $d \geq 3$ and its multi-degree
 $[m_1, \cdots, m_c]$ satisfies $d= m_1 < d+2 \leq m_2 \leq \cdots \leq m_c$.

Then any VMRT-structure on a uniruled projective manifold which is isotrivial of type $Z \subset \BP V$ is locally flat.
\end{theorem}

In the proof, we need the vanishing $H^0(Z, T_Z(1)) =0$  to
relate (see Proposition \ref{p.5.2} and Corollary \ref{c.admissible}) the minimal rational curves to the natural Cartan coframe of the structure via the method of \cite{Hw10}. This vanishing does not hold for curves of multi-degrees $$[2], [3], [4], [2,2], [2,3], [2,2,2]$$ and for varieties
of dimension $\geq 2$ with multi-degree $[2]$, $[3]$ or $[2,2]$.
The case [2] of quadric hypersurfaces is covered by Mok's result in Example \ref{ex1}.
The case [3] of cubic hypersurfaces is also settled in \cite{Hw13} if $\dim V \geq 4,$ by using ideas different  from the method of \cite{Hw10}.
The restriction on the multi-degree $m_i$'s in (iii)
  is needed to get $\dim H^0(Z, TZ \otimes \Omega_Z \otimes \sO_Z(1)) = \dim V$.  It is likely that the result of Theorem \ref{t.Main}
still holds for other cases. But one has to come up with some new ideas to handle these cases.

  Theorem \ref{t.Main} is proved by checking certain cohomological conditions for complete intersections (Theorem \ref{t.cohomology} and Theorem \ref{t.Main2}).
  While some of these must have been known to the experts, some of them (for example, Theorem \ref{t.cohomology} (III) ,(IV) and
  Theorem \ref{t.Main2})  seem to be new and should be interesting as  purely algebro-geometric results on complete intersections.

\section{Coframes adapted to characteristic connections}\label{s.coframe}
Most of the results in this section are contained in \cite{Hw10}. But many of them
 are not explicitly stated in \cite{Hw10}, buried
   in the proofs of some propositions.
   For the reader's convenience, we will reproduce them here. Since it does not make sense to
   repeat all the arguments given in \cite{Hw10}, we take this opportunity to give an alternative presentation,
    using explicit computations with  respect to a chosen basis.  This presentation  should be more friendly to readers with background
     in differential geometry. People preferring a basis-free, invariant  approach should look at the presentation in \cite{Hw10}.

\begin{notation}\label{n} For a complex manifold $M$, we will write $T(M)$ to denote the tangent bundle of $M$,
 but sometimes use $T_M $ to simplify the notation.   Let $V$ be a $\C$-vector space and let $\BP V$ be its projectivization,
i.e., the set of 1-dimensional subspaces of $V$. Given a projective submanifold $Z \subset \BP V$,
 the
affine cone of $Z$ will be denoted by $\widehat{Z} \subset V$. For
a point $\alpha \in Z$, the affine tangent space of $Z$ at
$\alpha$ is
$$T_{\alpha}(\widehat{Z}):= T_u(\widehat{Z}) \subset V \mbox{ for a
non-zero vector } u \in \widehat{\alpha}.$$ This is independent of the
choice of $u$. There is a canonical identification $T_{\alpha}(Z)
= \Hom(\widehat{\alpha}, T_{\alpha}(\widehat{Z})/\widehat{\alpha}).$
\end{notation}

\begin{definition}\label{d.cone}  Let $M$ be a complex manifold.  A {\em smooth cone structure} on $M$ is
a submanifold $\sC \subset \BP T(M)$ such that  the projection $\varpi: \sC \to M$ is a smooth morphism
with irreducible fibers.
For each point $x \in M$, the fiber $\varpi^{-1}(x)$ will
be denoted by $\sC_x $ and the union of the affine cones
$\widehat{\sC}_x$ will be denoted by $\widehat{\sC} \subset T(M)$.
For a point  $\alpha \in \sC$,
denote by ${\rm d} \varpi_{\alpha}: T_{\alpha}(\sC) \to T_x(M), x=
\varpi(\alpha),$ the differential of $\varpi$ at $\alpha$. We have
three subspaces of $T_{\alpha}(\sC)$ defined by
$$\sV_{\alpha} := {\rm d}
\varpi_{\alpha}^{-1}(0), \;  \sT_{\alpha} := {\rm d}
\varpi_{\alpha}^{-1}(\widehat{\alpha}), \;  \sP_{\alpha} := {\rm d}
\varpi_{\alpha}^{-1}(T_{\alpha}(\widehat{\sC}_x)).$$ This gives three
vector subbundles  $\sV \subset \sT \subset \sP$ of  $T(\sC),$
i.e., three natural distributions on $\sC$. The distribution $\sV
\subset T(\sC)$ is integrable and coincides with the relative
tangent bundle of $\varpi$.
\end{definition}

\begin{definition}\label{d.connection} For a smooth cone structure $\sC \subset \BP T(M)$,
 a  line subbundle $\sF \subset T(\sC)$ is called a {\em
conic connection}  if $\sF \subset \sT$ and $\sF \cap \sV =0$, i.e., it
splits the exact sequence \begin{equation} \label{split} 0
\longrightarrow \sV \longrightarrow \sT \longrightarrow \sT/\sV
\cong \sO(-1)|_{\sC} \longrightarrow 0\end{equation} where $\sO(-1)$
denotes the relative tautological line bundle on $\BP T(M)$, in other words, it is the line
subbundle of $\varpi^* T(M)$,
the fiber of which at $\alpha \in \BP T(M)$ is given by $\sO(-1)_{\alpha} = \widehat{\alpha}.$
\end{definition}

The following is proved in Proposition 1 of \cite{HM04}.

\begin{proposition}\label{p.sP} In Definition \ref{d.connection}, regarding the subbundles  of
$T(\sC)$ as sheaves of vector fields on $\sC$, we have $\sP =
[\sF, \sV]$. In particular, the bracket operation on vector fields belonging to $\sF$ and $\sV$
induces an isomorphism of vector bundles $$
\Upsilon: \sF \otimes \sV  \ \to \ \sP/\sT.$$
\end{proposition}

\begin{definition}\label{d.characteristic} A conic connection $\sF \subset T(\sC)$ is a {\em
characteristic connection} if  for any local section $v$ of $\sP$
and any local section $w$ of $\sF$, both regarded as local vector
fields on the manifold $\sC$, the Lie bracket $[v, w]$ is a local
section of $\sP$ again. \end{definition}

When  $\sC = \BP T(M)$, any connection is a characteristic
connection.  This is exceptional, as we have the following proposition, which follows from
Theorem 3.1.4 of \cite{HM99} because an irreducible nonsingular  projective variety with degenerate Gauss map must be a linear
subspace.

\begin{proposition}\label{p.unique} If $\sC_x \subset \BP T_x(M)$ is not a linear subspace
for a general $x \in M$, then  a characteristic connection is
unique if it exists.
\end{proposition}

\begin{definition}\label{d.isotrivial2}
Let $Z \subset \BP V$ be an irreducible nonsingular nondegenerate projective variety.
Let $\sC \subset \BP T(M)$ be a smooth cone structure.
We say that $\sC$ is {\em isotrivial of type} $Z \subset \BP V$ if the fiber $\sC_x \subset \BP T_x(M)$ is
isomorphic to $Z \subset \BP V$ by a linear isomorphism $T_x(M) \cong V$ for each point $x \in M$.
\end{definition}

The following is Proposition 5.2 of \cite{Hw10}.

\begin{proposition}\label{p.5.2}
Let $\sC \subset \BP V$ be a smooth cone structure which is isotrivial of type $Z \subset \BP V$. If $H^0(Z, T_Z(1)) =0$, then
there exists a unique conic connection on $\sC$. \end{proposition}

Now we introduce Cartan's coframe method.

\begin{definition}\label{d.coframe}
Let $M$ be a complex manifold of dimension $n$. Fix a vector space
$V$ of dimension $n$. A $V$-valued 1-form $\omega$ on $M$ is
called a {\em coframe} if for each $x \in M$, the homomorphism
$\omega_x : T_x(M) \rightarrow V$ is an isomorphism.
 Fix a basis $e_1, \ldots, e_n$ of $V$. Then $\omega$ can be written as
$$\omega = \theta^1 e_1 + \cdots + \theta^n e_n$$ for some $1$-forms $\theta^1, \ldots, \theta^n$ on $M$.
 A coframe $\omega$ is {\em closed} if $d \theta^i =0$ for all $i$. A coframe $\omega$ is {\em
conformally closed} if for any point $x \in M$, there exist a
neighborhood $x \in U \subset M$  and a non-vanishing  function $f
$ on $U$ such that $f \omega $ is closed on $U$. The closedness (resp. conformal closedness) of a coframe is independent of the choice of a basis of $V$. \end{definition}

\begin{definition}\label{d.sigma}
Given a coframe $\omega$, let
 $T^i_{jk}$ be the functions on $M$ defined by $$ d \theta^i = \sum^n_{j,k =1} T^i_{jk} \theta^j \wedge \theta^k, \mbox{ satisfying } T^i_{jk} = - T^i_{kj}.$$ The $\Hom(V \wedge V, V)$-valued function
$\sigma^{\omega}$ defined by $\sigma^{\omega}(e_j, e_k) = \sum_{i=1}^n T^i_{jk} e_i$ is called the {\em structure function}
of the coframe $\omega$.
Denote by $\frac{\partial}{\partial \theta^1}, \ldots, \frac{\partial}{\partial \theta^n}$ the vector fields on $M$ defined by $$\langle \frac{\partial}{\partial \theta^i}, \theta^j \rangle = \delta_{ij}$$ where $\delta_{ij} = 0$ if $i \neq j$ and $\delta_{ii} =1$. Then $$ \left[ \frac{\partial}{\partial \theta^j}, \frac{\partial}{\partial \theta^k} \right] = 2 \sum_{i} T^i_{kj} \frac{\partial}{\partial \theta^i}.$$ \end{definition}

\begin{definition}\label{d.var} Denote by $\pi: T(M) \to M$ the natural projection to $M$. Then $\vartheta^i = \pi^* \theta^i$ are 1-forms on the complex manifold $T(M)$. The 1-form $\theta^i$ on $M$ can be viewed as a holomorphic function on $T(M)$, which we denote by $\lambda^i$. In other words, the value of the function $\lambda^i$ at a point $v \in T(M)$ is just $\lambda^i(v) = \theta^i(v).$
So we have a collection of 1-forms on $T(M)$  $$\vartheta^1, \ldots, \vartheta^n, d\lambda^1, \ldots, d\lambda^n.$$
Denote by $\frac{\partial}{\partial \vartheta^i}$ and $\frac{\partial}{\partial \lambda^j}$ the holomorphic vector fields on $T(M)$ defined by $$\langle \frac{\partial}{\partial \vartheta^i}, \vartheta^j \rangle = \delta_{ij} =
\langle \frac{\partial}{\partial \lambda^i}, \lambda^j \rangle \mbox{ and } \langle \frac{\partial}{\partial \vartheta^i}, d \lambda^j \rangle =0=\langle \frac{\partial}{\partial \lambda^i}, \vartheta^j \rangle.$$
\end{definition}

The proof of the following lemma is straight forward.

\begin{lemma}\label{l.basis}
Using the notation of Definition \ref{d.var} and writing $T^i_{jk}$ in place of the pullback $\pi^* T^i_{jk}$ for simplicity, we have $$ d \vartheta^i = \sum_{j,k} T^i_{jk} \vartheta^j \wedge \vartheta^k $$
$$ \left[ \frac{\partial}{\partial \vartheta^j}, \frac{\partial}{\partial \vartheta^k} \right]  = 2 \sum_{i} T^i_{kj} \frac{\partial}{\partial \vartheta^i} $$ $$\left[ \frac{\partial}{\partial \lambda^j},  \frac{\partial}{\partial \lambda^k} \right]  = \left[ \frac{\partial}{\partial \vartheta^j}, \frac{\partial}{\partial \lambda^k} \right]  = 0. $$ \end{lemma}

\begin{definition}\label{d.gamma}
Given a coframe $\omega$ on $M$, the vector field $\gamma$ on $T(M)$ defined by $$\gamma := \lambda^1 \frac{\partial}{\partial \vartheta^1} + \cdots + \lambda^n \frac{\partial}{\partial \vartheta_n}$$ is called the {\em geodesic flow} of the coframe $\omega$. It is easy to check that $\gamma$ is determined by $\omega$, independent of the choice of the basis. \end{definition}

The following is immediate from the definition.

\begin{lemma}\label{l.easy} For a point $ u \in T(M),$ denote by
$\gamma_u \in T_{u}(T(M))$ the value of $\gamma$ at $u$. Then
$d \pi_u (\gamma_u) = u.$ \end{lemma}

The following lemma is contained in the proof of Proposition 5.6 of \cite{Hw10}.

\begin{lemma}\label{l.gamma}
 Let $\omega$ be a coframe on $M$ and let $\tilde{v}$ be a germ of vector fields near a point  $u \in T(M)$ of the form $$\tilde{v} = h_1(\lambda) \frac{\partial}{\partial \lambda^1} + \cdots + h_n(\lambda) \frac{\partial}{\partial \lambda^n}$$ where $h_i(\lambda) = h_i(\lambda^1, \ldots, \lambda^n)$ is a germ of holomorphic functions in $n$-variables.
Then $$ d \pi_u ([[\tilde{v}, \gamma],
\gamma]_u)  = 2 \omega_x^{-1}(\sigma_x^{\omega}(\omega_x(u), \omega_x(v)))$$ where $x = \pi(u)$, $\omega_x: T_x(M)\to V$ is the restriction of $\omega$ at $x$ and $\sigma^{\omega}_x \in \Hom(V \wedge V, V)$ is the value of the structure function $\sigma^{\omega}$ at $x$.  \end{lemma}

\begin{proof}
 Note that $$[\tilde{v}, \gamma] = [ \sum_i h_i(\lambda) \frac{\partial}{\partial \lambda^i}, \sum_j
 \lambda^j \frac{\partial}{\partial \vartheta^j} ] = \sum_i h_i(\lambda) \frac{\partial}{\partial \vartheta^i}.$$  Applying Lemma \ref{l.basis},  we have \begin{eqnarray*} [[\tilde{v}, \gamma], \gamma] &=& [ \sum_j h_j(\lambda) \frac{\partial}{\partial \vartheta^j},
  \sum_k \lambda^k \frac{\partial}{\partial \vartheta^k}] \\  &=&    2 \sum_{i,j,k} h_j(\lambda) \lambda^k T^i_{kj} \frac{\partial}{\partial \vartheta^i}\\
   &=& 2 \omega_x^{-1}(\sigma_x^{\omega}(\omega_x(u), \omega_x(v))).  \end{eqnarray*}
\end{proof}

\begin{definition}\label{d.iso}  Let $Z \subset
\BP V$ be a fixed projective submanifold and  let $\sC \subset \BP
T(M)$ be a $Z$-isotrivial cone structure on $M$ as in Definition \ref{d.cone}.    A
coframe $\omega$ on $M$ is {\em adapted to the cone structure}
$\sC \subset \BP T(M)$ if $\omega_x: T_x(M) \to V$ sends $\sC_x
\subset \BP T_x(M)$ to $Z \subset \BP V.$ Any $Z$-isotrivial cone
structure has an adapted coframe if we replace $M$ by a sufficiently small open
subset.
\end{definition}

\begin{definition}\label{d.flat}
A smooth cone structure $\sC \subset \BP T(M)$ is {\em locally flat }
if  there are holomorphic coordinates $(x^1, \ldots, x^n)$ on a non-empty open subset $U \subset X$ such that the coordinate vector fields
  $\frac{\partial}{\partial x^1}, \ldots,  \frac{\partial}{\partial x^n}$ are $\sC$-preserving, in the sense described in Definition \ref{d.vmrt}. \end{definition}

The following is Proposition 4.5 in \cite{Hw10} and the proof is straight-forward.

\begin{proposition}\label{p.4.5}
 A $Z$-isotrivial cone structure on a manifold $M$ is locally flat if and only
if each point of $M$ has an open neighborhood over which one can find
 a conformally closed  adapted coframe. \end{proposition}

The next proposition is Proposition 4.6 in \cite{Hw10}.

\begin{proposition}\label{p.geodesic}
Let $\omega$ be a coframe adapted to a $Z$-isotrivial cone
structure $\sC \subset \BP T(M).$ Regard $T(M)$ as a complex manifold.
 Then  the geodesic flow $\gamma$ is tangent to the affine cone $\widehat{\sC}
\subset T(M).$
 \end{proposition}

 \begin{proof}
 Let  $(t^1,\ldots, t^n)$ be the coordinates
 on $V$ dual to the basis $\{ e_1, \ldots, e_n \}$.
 Let $$\{ f_a(t^1, \ldots, t^n), 1 \leq a \leq k \}$$ be the homogeneous polynomials generating the ideal $I_Z$ defining the
 projective variety $Z \subset \BP V.$
   Since $\omega$ is adapted, the variety $\sC \subset \BP T(M)$ is defined as the zero locus given by
   $$f_a(\lambda^1, \ldots, \lambda^n) = 0,  \ 1 \leq a \leq k.$$ Then $$\gamma ( f_a)  = (\lambda^1 \frac{\partial}{\partial \vartheta^1} + \cdots + \lambda^n \frac{\partial}{\partial \vartheta_n}) f_a( \lambda^1, \ldots, \lambda^n) = 0.$$ Thus
  the vector field  $\gamma$ is tangent to $\sC$. \end{proof}

\begin{definition}\label{d.geodesic}
By Lemma \ref{l.easy} and  Proposition \ref{p.geodesic}, the image of the
vector field $\gamma$ in $\sC$ spans a foliation $\sF^{\omega}
\subset T(\sC)$ which is a conic connection for the cone structure
$\sC$. This conic connection $\sF^{\omega}$ is called the {\em
geodesic connection} of the adapted coframe $\omega$. For an alternate way
to define $\sF^{\omega}$, see Proposition 3.11 in \cite{Hw13}.
\end{definition}

\begin{definition}\label{d.adpated}
Let $\sC \subset \BP T(M)$ be a cone structure with a conic
connection $\sF \subset \sT \subset T(\sC)$. Then a coframe
$\omega$ on $M$ is said to be {\it adapted to the conic
connection} $\sF$ if it is adapted to $\sC$ and its geodesic
connection agrees with the given conic connection, i.e.,
$\sF^{\omega}= \sF$.
\end{definition}

\begin{definition}\label{d.Xi}
Given a projective submanifold
 $Z \subset \BP V$, define  \begin{eqnarray*} \Xi_Z & := & \{
\sigma: V \wedge V \to V, \  \sigma(u, v) \in T_u(\widehat{Z})  \mbox{
if } u \in \widehat{Z} \mbox{ and } v \in T_{u}(\widehat{Z}) \}\\
 \Xi_{V} & = & \{
\sigma: V \wedge V \to V, \; \sigma(u, v) \in \C u + \C v \mbox{
for all  } u, v \in V \}.\end{eqnarray*} Note that $\Xi_V \subseteq \Xi_Z$.  From Proposition 3.3 of
\cite{Hw10},  the subspace $\Xi_{V} \subset \Hom (V
\wedge V, V)$ is isomorphic to the dual space $V^{\vee}$ by the
natural contraction homomorphism $V^{\vee} \to \Hom(V \wedge V,
V)$.
\end{definition}

The following is Theorem 3.4 of \cite{Hw10}.

\begin{proposition}\label{p.theorem3.4}
A coframe $\omega $ on a manifold $M$ of dimension $\geq 3$ is
conformally closed if and only if the structure function
$\sigma^{\omega}$ takes values in $\Xi_{V} \subset \Hom(V
\wedge V, V)$. \end{proposition}

The next theorem is a refinement of Proposition 5.6 in
\cite{Hw10}.

\begin{theorem}\label{t.connection}
Let $\sC \subset \BP T(M)$ be a $Z$-isotrivial cone structure and
let $\omega$ be an adapted coframe with the structure function
$\sigma^{\omega}$. If the geodesic connection $\sF^{\omega}$ is a
characteristic connection, then $\sigma^{\omega}$ takes values in
$\Xi_Z$.  \end{theorem}

\begin{proof} Denoting by $0_M \subset T(M)$ the zero section, let $\psi: \hat{\sC} \setminus 0_M \to \sC$ be the natural projection.
The geodesic flow $\gamma$ of Proposition \ref{p.geodesic}
satisfies $d \psi(\gamma) \subset \sF^{\omega}$ by Definition
\ref{d.geodesic}.

As the coframe is adapted to $\sC $, for a
point $u \in \hat{\sC}_x$ and a vector $v \in T_{u}(\hat{\sC_x})$
satisfying $d \psi(v) \neq 0$, we can choose
a local vector field $\tilde{v}$   in a neighborhood of $u$ in $\hat{\sC}$, which has the
  the form $$\tilde{v} = h_1(\lambda) \frac{\partial}{\partial \lambda^1} + \cdots + h_n(\lambda) \frac{\partial}{\partial \lambda^n}.$$ Thus we can apply Lemma \ref{l.gamma}.
 to see that
$$\sigma_x^{\omega}(u,v) = d \pi_u([[ \tilde{v}, \gamma],
\gamma]_u).$$
Since $\sF^{\omega}$ is a characteristic connection,
the local vector field $[[ \tilde{v}, \gamma],
\gamma]$ is a section of $\sP$.  It follows that
$\sigma_x^{\omega}(u,v)$ has value in $T_u(\widehat{Z})$ modulo $\C u$. Thus
$\sigma^{\omega}$ takes values in $\Xi_Z$.
\end{proof}

\begin{corollary}\label{c.admissible}
Let $Z \subset \BP V$ be a submanifold with $\Xi_Z= \Xi_{
V}$.  Let $M$ be a complex manifold of dimension $\geq 3$ and let $\sC \subset \BP T(M)$ be a $Z$-isotrivial cone
structure with a characteristic connection $\sF$. Then $\sC$ is
locally flat if at least one of the following two conditions holds.
\begin{itemize}  \item[(1)] There
exists a coframe adapted to the connection $\sF$.   \item[(2)] $H^0(Z, T(Z) \otimes \sO(1)) =0$.\end{itemize} \end{corollary}

\begin{proof}
Assume that (1) holds. Then the geodesic connection of the adapted coframe
coincides with the characteristic connection $\sF$. Thus $\sC$ is locally flat
from Proposition \ref{p.4.5}, Proposition \ref{p.theorem3.4} and
Theorem \ref{t.connection}.

Assume that (2) holds. Then a $Z$-isotrivial cone
structure has a unique conic connection from Proposition \ref{p.5.2}.
Thus any coframe adapted to $\sC$ is adapted to the
 connection $\sF$ and the condition (1) is satisfied.
\end{proof}

When the smooth cone structure is a VMRT-structure, the existence of a characteristic connection is automatic by the following result of
Proposition 6.1 \cite{Hw10}, which is a reformulation of Proposition 3.1.2 of \cite{HM99}.

\begin{proposition}\label{p.6.1}
Let $X$ be a uniruled projective manifold and let $\sC \subset \BP T(X)$ be the VMRT-structure defined by a family $\sK$ of minimal rational curves on $X$. Assume that there exists an open subset $M \subset X$ such that the restriction $\sC|_M \subset \BP T(M)$ is a smooth cone structure. Then it has a characteristic connection given by tangent vectors of members of $\sK$. \end{proposition}

By Proposition \ref{p.6.1} and Corollary \ref{c.admissible}, we have
\begin{corollary} \label{c.criteria}
Let $Z \subset \BP V$ be a smooth irreducible subvariety of positive dimension. Assume that
$H^0(Z, T(Z) \otimes \sO_Z(1)) = 0$ and $\Xi_Z = \Xi_V$.  Then any VMRT-structure on a uniruled projective manifold which is isotrivial of type
$Z \subset \BP V$ is locally flat.
\end{corollary}

\section{Outline of the proof of Theorem \ref{t.Main}}

\begin{definition}\label{d.Xi2} Let $Z \subset \BP V$ be a nonsingular projective variety.
We will denote by $\Omega_Z (= \Omega^1_Z)$ the cotangent bundle of $Z$.
For a vector $v \in T_{\alpha}(\widehat{Z})$ in Notation \ref{n}, denote  by $[v] \in T_{\alpha}(\widehat{Z})/\widehat{\alpha}$ its class modulo $\widehat{\alpha}.$ By the canonical identification $T_\alpha(Z) = {\rm Hom}(\widehat{\alpha}, T_\alpha(\widehat{Z})/\widehat{\alpha})$,
the fiber of $T_Z \otimes \sO_Z(-1)$ at $\alpha$ is identified with $T_\alpha(\widehat{Z})/\widehat{\alpha}$.

 \begin{itemize}
\item[(i)] Note that the fiber of $T_Z \otimes \Omega_Z\otimes \sO(1)$ at a
point $\alpha \in Z$ is naturally isomorphic to

 $${\rm
Hom}\left( \widehat{\alpha} \otimes (T_{\alpha}(\widehat{Z})/\widehat{\alpha}), \   T_{\alpha}(\widehat{Z})/\widehat{\alpha}) \right).$$      For an element  $\sigma \in \Xi_Z$, define  $\zeta(\sigma) \in H^0(Z, T_Z \otimes \Omega_Z\otimes \sO(1))$
such that its value at $\alpha \in Z,$
$$\zeta(\sigma)_{\alpha} \in {\rm
Hom}\left(\widehat{\alpha} \otimes (T_{\alpha}(\widehat{Z})/\widehat{\alpha}), \  T_{\alpha}(\widehat{Z})/\widehat{\alpha}\right),$$
satisfies $\zeta(\sigma)_{\alpha} (u\otimes  [v]) = [\sigma(u, v)] $ for any $u \in \widehat{\alpha}$ and $v \in T_{\alpha}(\widehat{Z})$.
This defines  a homomorphism
$$
\zeta: \Xi_Z \to H^0(Z, T_Z \otimes \Omega_Z\otimes \sO(1)).
$$
 Define $\Xi'_Z:= {\rm Ker}(\zeta)$, i.e.,  $$ \Xi'_Z :=
\{ \sigma \in \Hom( V \wedge V, V), \ \sigma(\widehat{\alpha},v) \subset \widehat{\alpha} \mbox{ if } \widehat{\alpha} \subset  \widehat{Z} \mbox{ and } v \in T_{\alpha}(\widehat{Z})\}.$$
\item[(ii)] Note that  the fiber of $\Omega_Z \otimes \sO(1)$ at a point
$\alpha \in Z$ is identified with ${\rm Hom}(T_{\alpha}(\widehat{Z})/\widehat{\alpha}, \C)$.
For $\sigma \in \Xi'_Z$, let $\eta_{\sigma} \in H^0(Z, \Omega_Z \otimes \sO(1))$ be the section whose value at $\alpha \in Z$,
$$\eta_{\sigma, \alpha} \in {\rm Hom}(T_{\alpha}(\widehat{Z})/\widehat{\alpha}, \C)$$
satisfies $\eta_{\sigma, \alpha}([v]) u = \sigma (u, v)$ for any $u \in \widehat{\alpha}$ and  $v \in T_{\alpha}(\widehat{Z})$. This defines a homomorphism
$$
\eta: \Xi'_Z \to H^0(Z, \Omega_Z \otimes \sO(1)).
$$
Define $\Xi^{0}_Z :=
{\rm Ker}(\eta)$, i.e., $$\Xi^{0}_Z = \{ \sigma \in \Hom( V \wedge V, V), \ \sigma(u,v) =0 \mbox{ if } u \in \widehat{Z} \mbox{ and } v \in T_u(\widehat{Z})\}.$$
Note that $\Xi^{0}_Z =0$ if and only if $Z$ is tangentially nondegenerate, i.e., the variety ${\rm TanLines}(Z) \subset \BP (\wedge^2 V)$ of tangent lines to $Z \subset \BP V$ is nondegenerate in $\BP(\wedge^2 V)$.
 \end{itemize} \end{definition}


\begin{theorem}\label{t.cohomology}
Let $Z \subset \BP V$ be a positive-dimensional nonsingular nondegenerate complete intersection. Then \begin{itemize}
\item[(I)] $H^0(Z, \Omega_Z(1)) =0$ unless $Z$ is a curve.
\item[(II)] $H^0(Z, T_Z(1)) \neq 0$ if and only if
\begin{itemize}
\item[(II-a)]  either $Z$ is a curve whose multi-degree is one of the following
$$[2], [3], [4], [2,2], [2,3], [2,2,2].$$
\item[(II-b)] or $\dim Z \geq 2$, whose multi-degree is one of the following
$$[2], [3], [2, 2]$$
\end{itemize}
\item[(III)] $\Xi_Z^{0} =0$, i.e., $Z$ is tangentially nondegenerate.
 \item[(IV)] $\Xi'_Z \neq 0$ if and only if $Z$ is a plane conic.  \end{itemize}
\end{theorem}

Now Theorem \ref{t.Main} follows from Corollary \ref{c.criteria}, Theorem \ref{t.cohomology} (II) and the following result:

\begin{theorem} \label{t.Main2}
Let $Z \subset \BP V$ be a smooth non-degenerate complete intersection.
Assume further that $Z$ satisfies one of the following:

(i) $Z$ is a  curve  of degree $\geq 3$;

(ii) $Z \subset \BP V$ is covered by lines and not a quadric hypersurface;

(iii) $Z$ is contained in a smooth hypersurface of degree $d \geq 3$ and its multi-degree
 $[m_1, \cdots, m_c]$ satisfies $d= m_1 < d+2 \leq m_2 \leq \cdots \leq m_c$.

Then  $\Xi_Z  = \Xi_V$.
\end{theorem}

The proofs of Theorem \ref{t.cohomology} and Theorem \ref{t.Main2} will be given in Sections \ref{s.(i)} -- \ref{s.XiZV}.
At several points in the proofs, we will use the linear normality $H^0(Z, \sO(1)) \cong V^{\vee}$ of complete intersections, without explicitly mentioning it.

\begin{remark} In the proof of Proposition 5.7 in \cite{Hw10}, the two vector spaces $\Xi'_Z $ and $ \Xi^0_Z $ were erroneously mixed up. To fix this error, one has to add the condition $H^0(Z, \Omega_Z(1)) =0$ in the statement of Theorem 1.1  and Proposition 5.7 in \cite{Hw10} (consequently, also for Theorem 1.10 in \cite{Hw13}). Since this condition is  satisfied when $Z$ is a hypersurface, this does not affect the other results in \cite{Hw10} and \cite{Hw13}.
\end{remark}

\section{Proof of Theorem \ref{t.cohomology} (i) and (ii)}\label{s.(i)}

Theorem \ref{t.cohomology} (I) is immediate from
the following lemma, which is proved in  \cite{B76} (Satz 3) and \cite{B77} (Satz 6).
\begin{lemma}\label{l.Bruck}
Let $Z \subset \BP V$ be a smooth non-degenerate complete intersection. Assume $1 \leq r \leq \dim Z -1$, then $H^0(Z, \Omega_Z^r(p))=0$ for  $p \leq r$.
\end{lemma}

Let us prove (II). For  a smooth complete intersection curve $Z \subset \BP^N$ of multi-degree $[m_1, \cdots, m_{N-1}]$, let $m=\sum_{i=1}^{N-1} m_i$. Then
$H^0(Z, T_Z(1)) = H^0(Z, \sO(N+2-m))$ which is non-zero if and only if $N+2\geq m$.
This gives exactly the multi-degrees in (II-a).

Now assume $n:=\dim Z \geq 2$.
Note that  $$H^0(Z, T_Z(1)) \simeq H^0(Z, \Omega_Z^{n-1}(1) \otimes (\Omega_Z^{n})^*) = H^0(Z, \Omega_Z^{n-1}(N+2-m)).$$ By Lemma \ref{l.Bruck}, we have
$H^0(Z, T_Z(1))=0$ provided that $N+2-m \leq n-1$, or equivalently, $\sum_i (m_i-1) \geq 3$.
This is the case except for the types $[2]$, $[3]$ and $[2, 2]$.
Since $H^0(Z, \Omega_Z^r(r+1)) \neq 0$ by \cite{B77} (Satz 7) for $1 \leq r \leq n-1$, we see that  $H^0(Z, T_Z(1)) \neq 0$ for these three cases, proving (II-b).

\section{Proof of Theorem \ref{t.cohomology} (III)}\label{s.(iv)}

\begin{notation}\label{d.normal}
For a nonsingular projective variety $Z \subset \BP V$,
the normal space of $Z$ at $\alpha \in Z$ is $N_{Z,\alpha} = \Hom (\widehat{\alpha}, V/T_{\alpha}(\widehat{Z}))$ and its dual
$N^{\vee}_{Z,\alpha}$ is the conormal space. We will  denote by $N^{\vee}_Z$ the conormal bundle of $Z$. \end{notation}

\begin{definition}\label{d.lambda} Fix a nonzero element $\delta \in V^\vee$.
For $A \in \Xi^0_Z$, define an element $\lambda_\delta (A) \in H^0(Z, N_Z^{\vee}(2))$
by setting its value   at
 $\alpha \in Z$ to be
$$\lambda_\delta (A)  (\varphi \otimes
u^{\otimes 2}) := \delta \circ A (u, \widetilde{\varphi(u)})$$ where $u \in \widehat{\alpha}$ and
for a normal vector $\varphi \in N_{Z,u} = \Hom (\widehat{u},
V/T_u(\widehat{Z}))$,
we denote by $\widetilde{\varphi(u)} \in V$  a vector representing $\varphi(u)
\in V/T_{\alpha}(\widehat{Z})$. Note that the above
definition does not depend on the choice of $\widetilde{\varphi(u)}$
and depends only on $\varphi(u)$ because
 $A(\widehat{\alpha},T_{\alpha}(\widehat{Z})) =0$.
This defines a homomorphism $$\lambda_\delta :  \Xi^0_Z \to H^0(Z, N^{\vee}_Z(2)).$$
\end{definition}

\begin{definition}\label{d.chi}
 For an element $Q$ of
$$H^0(\BP V, \sI_Z(2)) = \{ Q \in {\rm Sym}^2 V^{\vee}, \ Q(u,u) =0 \mbox{ for all }u
\in \widehat{Z}\},$$ define  an element
$\chi(Q) \in H^0(Z, N^{\vee}_Z(2))$ by setting its value   at $\alpha \in Z$ to be
$$\chi(Q) (\varphi \otimes
u^{\otimes 2}) := Q (u, \widetilde{\varphi(u)})$$ where
 $u\in \widehat{Z}$ and
$\widetilde{\varphi(u)} \in V$ is as in Definition \ref{d.lambda}. Note that the above
definition does not depend on the choice of $\widetilde{\varphi(u)}$
and depends only on $\varphi$ because
 $Q(\widehat{\alpha},T_{\alpha}(\widehat{Z})) =0$ for all $\alpha \in Z$.
This defines a homomorphism $$\chi: H^0(\BP V,
\sI_Z(2)) \to H^0(Z, N^{\vee}_Z(2)).$$
\end{definition}

\begin{remark} As is well-known, the sheaf of local sections of the bundle $N_Z^{\vee}$ is just $\sI_Z/\sI_Z^2$.
The homomorphism $\chi$ comes from the natural map $\sI_Z \to \sI_Z/\sI^2_Z$.
\end{remark}

\begin{lemma}\label{l.tan}
In the setting of Definition \ref{d.lambda} and Definition \ref{d.chi},
we have ${\rm Im}(\lambda_\delta ) \cap {\rm Im}(\chi) =0$. \end{lemma}

\begin{proof}

For any $u \in \widehat{Z}$ and  any $w\in V$, we have an element $\varphi^w_u \in N_{Z, u}$ such that $\varphi^w_u(u) = w$.
Suppose $\lambda_\delta(A) = \chi(Q) $ for some $A \in \Xi^0_Z$ and $Q \in H^0(\BP V,
\sI_Z(2)).$
For all
$u \in \widehat{Z}$ and  $w \in V$,
 $$Q(u,w) = \chi(Q) (\varphi^w_u \otimes u^{\otimes 2}) =
 \lambda_\delta (A) (\varphi^w_u \otimes u^{\otimes 2}) = \delta \circ A (u,w).$$
 Thus $Q(u,w) =\delta \circ A(u,w)$ for
any $u \in \widehat{Z}$ and $w \in V$. Since $Z$ is
nondegenerate in $\BP V$, this implies that the symmetric form $Q$
and the antisymmetric form $\delta \circ A$ are equal as bilinear forms
on $V$, a contradiction unless $Q= \delta \circ A=0$.
\end{proof}

We recall the following two standard lemmata.
The first one can be found in p. 630 of \cite{BR} and the next one is easily checked by splitting the long exact sequence into short ones.

\begin{lemma}\label{l.Koszul} Let  $Z \subset \BP^N$ be a complete intersection of multi-degree $[m_1, \cdots, m_c].$
Put $n = \dim Z = N-c$.
Writing $m=m_1+\cdots+m_c$, we have the following
Koszul exact sequence
$$0 \to \sO_{\BP^N}(-m) \to \oplus_{i=1}^c \sO_{\BP^N}(-m+m_i) \to \cdots $$ $$ \cdots \to \oplus_{i=1}^c \sO_{\BP^N}(-m_i) \to \sO_{\BP^N} \to \sO_Z \to 0. $$
\end{lemma}

\begin{lemma}\label{l.vanish}
Let $0 \to \mathcal{F}_0 \to \mathcal{F}_1 \to \cdots \to
\mathcal{F}_m \to 0$ be an exact sequence of coherent sheaves on a
variety $X$. If $H^{q+j-1}(X, \mathcal{F}_{m-j})=0$ for
all $j \in \{1, 2, \cdots, m\}$, then $H^q(X, \mathcal{F}_m)=0$.
\end{lemma}

Now we can finish the proof of $\Xi^0_Z =0$ as follows.
Let  $[m_1, \ldots, m_c]$ be the
multi-degree of $Z$, satisfying $m_i
\geq 2$ for all $i$.   By  Lemma \ref{l.Koszul}, we have an exact sequence
$$
0 \to \sO_{\BP^N}(2-m) \to \cdots  \to \oplus_{i=1}^c \sO_{\BP^N}(2-m_i) \to  \sI_Z(2) \to 0
$$
By Lemma \ref{l.vanish}, we obtain $H^0(\BP V, \sI_Z(2)) \simeq H^0(\BP V, \oplus_{i=1}^c \sO_{\BP^N}(2-m_i))$.
Since the normal bundle $N_Z$ is isomorphic
to
$ \bigoplus_{i=1}^c \sO_Z(m_i), \;  m_i \geq 2,$ we see that
$$\chi: H^0(\BP V, \sI_Z(2)) \cong H^0(Z, N^{\vee}_Z(2)).$$
is an isomorphism.

Take any $A \in \Xi_Z^\circ$. By Lemma \ref{l.tan}, we have $\lambda_\delta(A)=0$ for any $\delta \in V^\vee$, namely $A(u, w)=0$ for all $u \in \hat{Z}, w \in V$, which implies that $A=0$ by the linear nondegeneracy of $Z$.
Thus $\Xi^0_Z=0$

\section{Proof of Theorem \ref{t.cohomology}  (iv)}\label{s.(v)}

Firstly, for the plane conic $$Z = \{x_1^2 + x_2^2+x_3^2 =0\} \subset \BP V, \ \dim V =3,$$
where $x_1, x_2, x_3$ are homogenous coordinates dual to a basis $(e_1, e_2, e_3)$ of $V$,
define $\sigma \in {\rm Hom}(\wedge^2V, V)$ by $$\sigma
(e_1 \wedge e_2)=e_3, \ \sigma (e_2 \wedge e_3) = e_1, \  \sigma (e_3
\wedge e_1) = e_2.$$  One can check that  $\sigma \in \Xi'_Z$. Hence $\Xi'_Z \neq 0$
for a plane conic.

Let us prove $\Xi'_Z =0$ when $Z$ is not a plane conic.
By Theorem \ref{t.cohomology} (III), the homomorphism $\eta: \Xi'_Z \to H^0(Z, \Omega_Z \otimes \sO(1))$ is injective.
Thus if $\dim Z \geq 2$, we have $\Xi'_Z =0$ by (I). So we may assume that $\dim Z =1$.
Let $[m_1, \ldots, m_{N-1}]$ be the multi-degree of the curve $Z \subset \BP V, N= \dim V -1$
and let $m = m_1 + \cdots + m_{N-1}$.

\begin{lemma}\label{l.projnormal}
If $Z$ is not a plane conic, then the homomorphism
$$H^0(Z, \Omega_Z \otimes \sO(m_i -2)) \otimes H^0(Z, \sO(1)) \to H^0(Z, \Omega_Z \otimes \sO(m_i -1))$$
is surjective for each $1 \leq i \leq N-1$.
\end{lemma}

\begin{proof}
By the projective normality of complete intersections, the homomorphism is surjective as long as
the degree of $$\Omega_Z \otimes \sO(m_i-2) \simeq \sO(m-N-1 + m_i -2) = \sO( m+ m_i -N -3) $$ is nonnegative.
From $m = m_1 + \cdots + m_{N-1} \geq 2(N-1)$, we have $m+ m_i -N-3 \geq 0$ unless $N=2$ and $m=m_1 =2$.
Thus the surjectivity holds unless $Z \subset \BP V$ is a plane conic.
\end{proof}

For the next lemma, we need the following definition.

\begin{definition}\label{d.gamma}
As in Definition \ref{d.Xi2}, we can identify the fiber of $\Omega_Z \otimes \sO(2)$ at a point $\alpha \in Z$  with ${\rm Hom} \left(\widehat{\alpha} \otimes (T_{\alpha}\widehat{Z}/\widehat{\alpha}),
\C \right)$.  For an element $\omega \in \wedge^2 V^{\vee}$ and $\alpha \in Z$, define $$\gamma(\omega)_{\alpha}
\in {\rm Hom}\left(\widehat{\alpha} \otimes (T_{\alpha}\widehat{Z}/\widehat{\alpha}),
\C \right) $$   by setting $$ \gamma(\omega) (u \otimes [v]) :=
\omega(u, v)
$$ for any $u \in \widehat{\alpha}$ and $[v] \in T_{\alpha}(\widehat{Z})/\widehat{\alpha}$ given by  $v \in T_{\alpha}(\widehat{Z})$. This defines a homomorphism
$$\gamma: \wedge^2 V^{\vee} \simeq H^0(\BP V, \Omega_{\BP V}(2)) \to
H^0(Z, \Omega_Z\otimes \sO(2)). $$ \end{definition}

\begin{lemma}\label{l.omega}
Let $$\iota_0: H^0(Z, \Omega_Z\otimes \sO(1)) \otimes H^0(Z, \sO(1)) \to H^0(Z, \Omega_Z \otimes \sO(2))$$ be
the natural tensor product homomorphism.
Then for any $\sigma \in \Xi'_Z$ and $s \in V^{\vee} = H^0(Z, \sO(1))$, there exists an element $\omega^{\sigma, s} \in \wedge^2 V^{\vee}$
satisfying $$\iota_0(\eta_{\sigma} \otimes s) = \gamma(\omega^{\sigma,s})$$ where $\gamma$ is as in Definition \ref{d.gamma} and $\eta_{\sigma}$ is as in Definition \ref{d.Xi2}(ii), the image
of $\sigma$ under the injection $\eta: \Xi'_Z \to H^0(Z, \Omega_Z\otimes \sO(1)).$
\end{lemma}

\begin{proof}
Note that the homomorphism $\iota_0$ restricted to the fibers at $\alpha \in Z$,
$${\rm Hom}(T_{\alpha}(\widehat{Z})/\widehat{\alpha}, \C) \otimes V^{\vee} \stackrel{\iota_{0}}{\longrightarrow}
{\rm Hom}\left(\widehat{\alpha} \otimes (T_{\alpha}(\widehat{Z})/\widehat{\alpha}), \C \right),$$ is given by
$$\iota_0( \psi \otimes s) (u \otimes [v]) = \psi([v]) \cdot s(u)$$ for any $\psi \in {\rm Hom}(T_{\alpha}(\widehat{Z})/\widehat{\alpha}, \C),$ $ u \in \widehat{\alpha}, v \in T_{\alpha}(\widehat{Z})$ and $s \in V^{\vee}$.
Define $\omega^{\sigma,s} \in \wedge^2 V^{\vee}$ by $\omega^{\sigma,s}(v_1, v_2) := s(\sigma(v_1, v_2)).$ Then for $u \in \widehat{Z}$ and $v\in T_u(\widehat{Z}),$ $$\gamma(\omega^{\sigma,s})(u \otimes [v])
= \omega^{\sigma, s}(u, v) = s (\sigma(u, v)) = $$ $$ = s( \eta_{\sigma}(v) u) = \eta_{\sigma}(v) \cdot s(u) =  \iota_0(\eta_{\sigma} \otimes s) (u \otimes [v]).$$
This proves the lemma.
\end{proof}

Recall the following from Bott's formula (\cite{OSS} p.8).

\begin{lemma} \label{l.Projective}

 $H^q(\BP^N, \Omega_{\BP^N}^r(p)) \neq 0 $ if and only if one of the following holds:
$$ \begin{cases}
(1) & q=0 \ \text{and} \ p \geq r+1; \\
(2) & q=r \ \text{and} \  p=0; \\
(3) & q=N \ \text{and} \  p \leq r-1-N.
\end{cases}
$$
\end{lemma}

\begin{lemma}\label{l.wedge}
Let $Z \subset \BP^N$ be a smooth complete intersection  of codimension $c$.
If  $r \geq {\rm max}\{1,t\}$, then
$H^0(Z, \Omega_{\BP^N|Z}^r(t)) = 0$.
\end{lemma}

\begin{proof}
From the exact sequence in Lemma \ref{l.Koszul}, we have the exact sequence
\begin{equation}\label{e.N}
 0 \to \Omega_{\BP^N}^{r}(t-m)  \to \oplus_{i=1}^c
\Omega_{\BP^N}^{r}(t-m + m_i) \to \cdots \end{equation} $$ \cdots \to \oplus_{i=1}^c
\Omega_{\BP^N}^{r}(t-m_i) \to \Omega_{\BP^N}^{r}(t) \to
\Omega_{\BP^N|Z}^{r}(t) \to 0. $$
Assume that $t \leq r$, then $H^0(\BP^N, \Omega_{\BP^N}^{r}(t))=0$ by
Lemma \ref{l.Projective}. Note that $m_{i_1} + \cdots + m_{i_r} \geq 2r >
r$, hence $t \neq m_{i_1} + \cdots + m_{i_r}$ for any
 indexes $i_1, \cdots, i_r$ whenever $ r \leq c$.
By Lemma \ref{l.Projective}, we have
\begin{equation*}
H^0(\BP^N, \Omega_{\BP^N}^{r}(t))=H^{1}(\BP^N,\oplus_{i=1}^c
\Omega_{\BP^N}^{r}(t-m_i)) = \cdots =H^{c}
(\BP^N, \Omega_{\BP^N}^{r}(t-m))=0.
\end{equation*}
Then the claim follows from Lemma \ref{l.vanish}.
\end{proof}

Now we are ready to  finish the proof of $\Xi'_Z =0$. We use the conormal exact sequence $0 \to N_Z^{\vee} \to \Omega_{\BP
V}|_Z \to \Omega_Z \to 0$ to obtain the following commutative
diagram  where all cohomology groups are taken over $Z$
$$
\begin{array}{ccccc}
H^0(\Omega_{\BP V}(1)|_Z) \otimes V^{\vee} & \longrightarrow & H^0(\Omega_Z(1))
\otimes V^{\vee} & \stackrel{\tau_1 \otimes {\rm Id}_{V^{\vee}}}{\longrightarrow}  &  H^1(N_Z^{\vee}(1)) \otimes V^{\vee} \\
\downarrow & &    {\iota_0} \downarrow & &     {\iota_1} \downarrow \\
 H^0(\Omega_{\BP V}(2)|_Z) &
\stackrel{r}{\longrightarrow} & H^0(\Omega_Z(2)) & \stackrel{\tau_2}{\longrightarrow} & H^1(N_Z^{\vee}(2)) \\
\uparrow & & \parallel & & \\
H^0(\BP V, \Omega_{\BP V}(2)) & \stackrel{\gamma}{\longrightarrow} & H^0(\Omega_Z(2)) .
\end{array}
$$
For any $\sigma \in \Xi'_Z$ and $s \in V^{\vee}$,  Lemma \ref{l.omega} says that
$\iota_0(\eta_\sigma \otimes s) $ belongs to the image of $r$.
Thus it is sent to zero by
$\tau_2$, implying  \begin{eqnarray}\label{e.tau} \iota_1 \left( \tau_1(\eta_{\sigma} \otimes s)  \right) &=& 0.\end{eqnarray}
Using $N_Z \cong \sum_{i=1}^{N-1} \sO(m_i)$ and Serre duality, we have
$$  \tau_1: H^0(Z,\Omega_Z(1)) \to  H^0(Z,\Omega_Z \otimes N_Z(-1))^{\vee} \cong \sum_{i=1}^{N-1} H^0(Z, \Omega_Z \otimes \sO(m_i-1))^{\vee}.$$ Write $$\tau_1(\eta_{\sigma}) = \sum_{i=1}^{N-1} \psi_i, \ \psi_i \in H^0(Z,\Omega_Z \otimes \sO(m_i-1))^{\vee}.$$ Then (\ref{e.tau}) implies $$0= \iota_1(\psi_i \otimes s) \in H^0(Z,\Omega_Z \otimes \sO(m_i-2))^{\vee}$$ for any $s \in V^{\vee}$ and $1 \leq i \leq N-1$.
Thus for any $t \in H^0(Z,\Omega_Z \otimes \sO(m_i-2)),$ we have
$$0 = \iota_1(\psi_i \otimes s) (t) = \psi_i (s \otimes t).$$ By Lemma \ref{l.projnormal}, this implies $\psi_i =0$ for all $1 \leq i \leq N-1$ which implies
 $\tau_1(\eta_{\sigma}) =0$. Since  $H^0(Z,\Omega_{\BP
V}(1)|_Z)=0$ by Lemma \ref{l.wedge}, the homomorphism $\tau_1$ is injective. We conclude that $\eta_{\sigma} =0$ and $\sigma =0$.

\section{Proof of Theorem  \ref{t.Main2}} \label{s.XiZV}

To prove Theorem \ref{t.Main2} (i) and (ii), we use the following corollary of Theorem \ref{t.cohomology} (III) and (IV).

\begin{corollary}\label{c.Xi}
Let $Z \subset \BP V$ be a positive-dimensional nonsingular nondegenerate complete intersection of positive codimension. Assume $Z \subset \BP V$
is not a plane conic, then the map $\zeta: \Xi_Z \to H^0(Z, T_Z \otimes \Omega_Z \otimes \sO(1))$ is injective.
In particular, if $ H^0(Z, T_Z \otimes \Omega_Z \otimes \sO(1)) = V^\vee$, then $\Xi_V = \Xi_Z$.
  \end{corollary}

The last statement follows from  $V^{\vee} \simeq \Xi_V \subseteq \Xi_Z \subseteq H^0(Z, T_Z \otimes \Omega_Z \otimes \sO(1)).$
Note that if $Z$ is a curve of degree $\geq 3$, then $H^0(Z, T_Z \otimes \Omega_Z \otimes \sO(1)) = H^0(Z, \sO(1))= V^\vee$.  Thus we obtain Theorem \ref{t.Main2} (i) by applying Corollary \ref{c.Xi}.

For the proof of Theorem \ref{t.Main2} (ii), we recall two lemmata.

\begin{lemma}\label{l.CoveredByLines}
Let $X \subset \BP V$ be a projective manifold covered by lines.
Assume that the VMRT $\sC_x \subset \BP T_x(X)$ at a general point $x \in X$
is nondegenerate and the Lie algebra $\aut(\widehat{\sC}_x) \subset {\rm End}(T_x(X))$ of infinitesimal automorphisms of the affine cone $\widehat{\sC}_x$ has dimension 1. Then
$H^0(X, T_X \otimes \Omega_X(1)) = H^0(X, \sO(1)).$ \end{lemma}

\begin{proof}
Assume that there exists a traceless element $A \in H^0(X, T_X \otimes \Omega_X(1))$.
For a general point $x \in X$, let $A_x \in {\rm End}(T_x(X))$ be a traceless endomorphism
representing the value of $A$ at $x$.
Let $C \subset X$ be a line through  $x \in X$. Then
$$T(X)|_C \cong \sO(2) \oplus \sO(1)^p \oplus \sO^q$$
for suitable nonnegative integers $p, q$ and if $\alpha \in \widehat{\sC}_x$
is a nonzero vector in $T_x(C) \subset T_x(X)$, we have $$T_{\alpha}(\widehat{\sC}_x)
\cong (\sO(2) \oplus \sO(1)^p)_x.$$
Restricting $A$ to $C$, we have $$T(X)|_C \cong \sO(2) \oplus \sO(1)^p \oplus \sO^q
  \stackrel{A|_C}{\longrightarrow}
T(X)\otimes \sO(1)|_C \cong \sO(3) \oplus \sO(2)^p \sO(1)^q.$$
Hence it sends $\sO(2)$ to $\sO(3) \oplus \sO(2)^p$.
This implies that for any  $\alpha \in \widehat{\sC}_x$,
the vector $A_x(\alpha)$ lies in $T_{\alpha}(\widehat{\sC}_x).$
This implies $A_x \in \aut(\widehat{\sC}_x)$. Since the latter is assumed to be $\C$, the traceless
endomorphism $A_x$ must be zero. Since this is so for a general $x$, we have $A=0$. \end{proof}

\begin{lemma} \label{l.aut}
Let $Z \subset \BP V$ be a smooth non-degenerate complete intersection. Then $\aut(\widehat{Z}) = \C$ unless $Z$ is a hyperquadric.
\end{lemma}
\begin{proof}
As $H^0(Z, T_Z) = H^0(Z, \Omega_Z^{n-1}(N+1-m))$,  by Lemma \ref{l.Bruck}
we have $H^0(Z, T_Z)=0$ if $N+1-m \leq n-1$ (namely $\sum_i (m_i-1) \geq 2$) and $n\geq 2$.
Now if $Z$ is a curve, then $H^0(Z, T_Z) = H^0(Z, \sO(N+1-m))$ is non-zero if and only if $N+1 \geq m$, which implies that $Z$ is
either  a plane cubic or of type $[2,2]$ in $\BP^3$. In the latter two cases, $Z$ is an elliptic curve and it is easy to see that
 $\aut(\widehat{Z}) = \C$.
\end{proof}

 The proof of Theorem \ref{t.Main2} (ii) can be obtained as follows. It is well-known (e.g. repeated applications of Example 1.4.2 in \cite{Hw01}) that if the multi-degree of $Z$ is $[m_1, \ldots, m_c]$, then the VMRT $\sC_x$ of $Z$ at a general point $x \in Z$ is a complete intersection of multi-degree
$$[2,3, \cdots, m_1, 2, 3, \cdots, m_2, \cdots, 2, 3, \cdots, m_c].$$
By Lemma \ref{l.aut}, $H^0(\sC_x, T_{\sC_x}) =0$ unless $Z$ is a hyperquadric. By Lemma \ref{l.CoveredByLines} and Corollary \ref{c.Xi},
we  have $\Xi_Z = \Xi_V$.

\medskip
It remains to prove Theorem \ref{t.Main2} (iii), which is equivalent to the following.

\begin{proposition} \label{p.XiZV}
Let $Y \subset \BP V$ be a smooth hypersurface of degree $d \geq 3$. Let $Z \subset Y$ be a smooth  complete intersection of $Y$  of
multi-degree $[m_1, \cdots, m_c]$ such that $m_c \geq \cdots \geq m_1 \geq d+2$. Then $\Xi_Z = \Xi_V$.
\end{proposition}

The proof is rather involved and will occupy the rest of the section.

\medskip
We start by introducing a notation. Let $T_k$ be the Young tableaux with $k+1$ boxes, which has two columns and the number of boxes in the first column is $k$ and in the second is $1$. We have the $T$-symmetrical tensor $\Omega^{T_k}$ in the sense of \cite{B1}.
Recall the following from Theorem 1 in \cite{B1}.

\begin{lemma} \label{l.P}

(i) $H^q(\BP^N, \Omega_{\BP^N}^{T_k}(p)) \neq 0$ if and only if one of the following holds:
$$ \begin{cases}
(1) & q=0 \ \text{and} \  p \geq k+3; \\
 (2) &  q=1 \ \text{and} \  p=k+1; \\
 (3) & q=k \ \text{and} \  p=1; \\
 (4) & q=N \ \text{and} \  p \leq k-N.
\end{cases}
$$

(ii) $\dim H^1(\BP^N, \Omega_{\BP^N}^{T_{N-1}}(N))  = N+1$.
\end{lemma}

We also recall the following standard results.

\begin{lemma} \label{l.P2}
(i)  $H^0(\BP^N, T_{\BP^N} \otimes \Omega_{\BP^N}(k)) \simeq H^0 (\BP^N, \sO_{\BP^N}(k))$ for any integer $k \leq 0$.

(ii) For $1 \leq q \leq N-2$ and any integer $k$, we have $H^q(\BP^N, T_{\BP^N} \otimes
\Omega_{\BP^N}(k))  =0 $  unless $(q, k) = (1, -1)$.

(iii)  $H^1(\BP^N, T_{\BP^N} \otimes \Omega_{\BP^N}(-1))
\simeq H^1(\BP^N, \Omega_{\BP^N}^{T_{N-1}}(N))$, which is of
dimension $N+1$.

(iv) $H^{N-1}(\BP^N, T_{\BP^N} \otimes \Omega_{\BP^N}(k)) =0 $ if $k \neq -N$.

\end{lemma}

\begin{proof}
Recall that $T_{\BP^N} \simeq  \Omega^{N-1}_{\BP^N} \otimes K_{\BP^N}^{\vee}$ and the subbundle of  traceless endomorphisms ${\rm ad}(T_{\BP^N}) \subset T_{\BP^N} \otimes \Omega^1_{\BP^N}$ satisfies
 ${\rm ad}(T_{\BP^N}) \simeq \Omega_{\BP^N}^{T_{N-1}} \otimes K_{\BP^N}^{\vee}$. Hence we have an exact sequence
 $$
 0 \to \sO_{\BP^N}(k) \to T_{\BP^N} \otimes \Omega_{\BP^N}(k) \to {\rm ad}(T_{\BP^N})(k) = \Omega_{\BP^N}^{T_{N-1}}(N+1+k) \to 0.
 $$
All the claims follow  from Lemma \ref{l.P} (ii) and (iii) applied to this exact sequence.
\end{proof}

To prove Proposition \ref{p.XiZV}, we need several results on the hypersurface $Y \subset \BP^N$.
To start with, we deduce the following result from Satz 2 of \cite{B74}.

\begin{lemma} \label{l.OmegaY}
(i) If $1 \leq r \leq N-2$, then $H^0(Y, \Omega_Y^r(p)) = 0$ for $p \leq r$.


(ii) If $q+r \neq N-1$ and $1 \leq q \leq N-2$, then
$$
\dim H^q(Y, \Omega_Y^r(p)) = \delta_{q, r} \cdot \delta_{p, 0}.
$$


\end{lemma}


The following lemma follows from Bott's formula (cf. Lemma \ref{l.Projective}) applied to
long exact sequence of cohomologies associated to the sequence $ 0 \to \Omega_{\BP^N}^r(p-d) \to \Omega_{\BP^N}^r(p) \to \Omega_{\BP^N |_Y}^r(p) \to 0$.
\begin{lemma} \label{l.OmegaPY}
(i)  If $p \leq r$ and $r \geq 1$, then $H^0(Y, \Omega_{\BP^N|_Y}^r(p)) =0$.

(ii) If $1 \leq q \leq N-2$, then
$$\dim H^q(Y, \Omega_{\BP^N |_Y}^r(p)) = \delta_{q,r} \cdot \delta_{p,0} + \delta_{q, r-1} \cdot \delta_{p, d}.$$
\end{lemma}

\begin{lemma}\label{l.PtoY}
Assume $p \leq -1$ and $0 \leq q \leq N-2$, then $H^q(Y, T_{\BP^N} \otimes \Omega_{\BP^N}(p)|_Y) = 0$
unless $(q, p) = (1, -1)$ or $(N-2, d-N)$.
\end{lemma}
\begin{proof}

From the exact sequence $0 \to  \sO_{\BP^N}(-d) \to \sO_{\BP^N} \to \sO_Y \to 0$ we have
$$
0 \to T_{\BP^N} \otimes \Omega_{\BP^N}(p-d) \to T_{\BP^N} \otimes \Omega_{\BP^N}(p) \to T_{\BP^N} \otimes \Omega_{\BP^N}(p)|_Y \to 0.
$$

Now the claim follows from Lemma \ref{l.P2}.
\end{proof}
\begin{lemma}\label{l.TPY}
Assume $p \leq -1$ and $0 \leq q \leq N-3$, then
$$H^q(Y, T_{\BP^N|_Y}(p)\otimes \Omega_{Y}) =0 $$  unless
$(q, p) = (1, -1)$ or $(N-3, 2d-N-1)$.
\end{lemma}
\begin{proof}

Note that $ T_{\BP^N|_Y}(p-d) \simeq \Omega_{\BP^N}^{N-1}(N+1+p-d)|_Y$, hence by Lemma \ref{l.OmegaPY},
if $p \leq -1$, then $H^q(Y, T_{\BP^N|_Y}(p-d)) = 0$ for all $0 \leq q \leq N-2$ unless $(q, p) = (N-2, 2d-N-1)$.

From the exact sequence $0 \to \sO_Y(-d) \to \Omega_{\BP^N|_Y} \to \Omega_Y \to 0$, we have
$$
0 \to T_{\BP^N|_Y}(p-d) \to T_{\BP^N} \otimes \Omega_{\BP^N}(p)|_Y \to T_{\BP^N|_Y}(p)\otimes \Omega_{Y} \to 0.
$$

Now the claim follows from Lemma \ref{l.PtoY}.
\end{proof}

\begin{proposition}\label{p.Y}
Assume $p \leq -1$. Then

(i) $ H^0(Y, T_Y \otimes \Omega_Y(p)) = 0$.

(ii) $ H^1(Y, T_Y \otimes \Omega_Y(p))=0$ if $p+d \leq 1$;

(iii) $ H^2(Y, T_Y \otimes \Omega_Y(p)) = 0 $ if $p+d \neq 0$;

(iv) For all $3 \leq q \leq N-4$,  $ H^q(Y, T_Y \otimes \Omega_Y(p)) =0$.

(v) $H^{N-3}(Y, T_Y \otimes \Omega_Y(p)) =0$ if $p \neq 2d-N-1$.
\end{proposition}
\begin{proof}
By Lemma \ref{l.OmegaY}, we have $H^0(Y, \Omega_Y(p+d))=0$ if $p+d \leq 1$ and
 $H^q(Y, \Omega_Y(p+d))=0$ for all $1 \leq q \leq N-3$ unless $(q, p) = (1, -d)$.

From the exact sequence $0 \to T_Y \to T_{\BP^N|_Y} \to \sO(d) \to 0$ we obtain
$$
0 \to T_Y \otimes \Omega_Y(p) \to T_{\BP^N|_Y} \otimes \Omega_Y(p) \to \Omega_Y(p+d) \to 0.
$$
Now the claim follows from Lemma \ref{l.TPY}.
\end{proof}

\begin{proposition} \label{p.dim=V}
For a smooth projective hypersurface $Y \subset \BP V$ of degree $\geq 3$, we have
$H^0(Y, T_Y \otimes \Omega_Y (1)) \simeq V^{\vee}$.
\end{proposition}
\begin{proof}
We have an exact sequence
$$ 0 \to \sO_Y(1)  \to T_Y \otimes \Omega_Y(1) \to ad(T_Y)(1) \to 0.$$

Note that $ad(T_Y)(1) \simeq \Omega_Y^{T_{N-2}}(N+2-d)$.
By Theorem 4(iii) \cite{BR}, we have $H^0(Y, \Omega_Y^{T_{N-2}}(N+2-d))=0$ if $d\geq 3$, which
implies the claim.
\end{proof}

Now, let  $Z \subset Y$ be a complete intersection of multi-degree $[m_1, \cdots, m_c]$, where $c = {\rm codim}_Y(Z) = N -1 - n$ with $n =\dim Z$.  We always assume $\dim Z \geq 2$, namely $c \leq N-3$.

\begin{lemma} \label{l.key}
Assume that $m_j \geq d$ for all $j$, then
 $$ H^0(Z, T_Y \otimes \Omega_Y(1) |_Z) \simeq H^0(Y, T_Y \otimes \Omega_Y(1))$$
\end{lemma}
\begin{proof}
To simplify the notation, put $\sF = T_{Y} \otimes \Omega_{Y}$.
By splitting the exact sequence in Lemma \ref{l.Koszul}, we have
$$
0 \to \sF(1-m) \to \oplus_{i=1}^c \sF(1-m+m_i) \to \cdots \to \oplus_{i=1}^c \sF(1-m_i) \to B \to 0
$$
$$
0 \to B \to  \sF(1) \to \sF(1)|_Z \to 0
$$
for some coherent sheaf $B$ on Y.  By Proposition \ref{p.Y},  we obtain the vanishing of the following cohomology groups
$$
H^0(Y,\sF(1-m_i)) = H^1(Y,\sF(1-m_i-m_j))=\cdots = H^{c-1}(Y, \sF(1-m)) = 0.
$$
Hence by Lemma \ref{l.vanish}, we get $H^0(Y, B)=0$. In a similar way, we get
$$
H^1(Y,\sF(1-m_i)) = H^2(Y, \sF(1-m_i-m_j))=\cdots = H^{c-1}(Y, \sF(1-m+m_i)) = 0
$$
Recall that $c \leq N-3$. If $c = N-3$, then $1-m \leq 1-d(N-3) < 2d-N-1$ since $N \geq 3$, hence we get
$H^{c}(Y,\sF(1-m))=0$ by Proposition \ref{p.Y}. This implies that $H^1(Y,B)=0$, concluding the proof.

\end{proof}

Now we are ready to finish the proof of Proposition \ref{p.XiZV}.

If $Z$ is a curve, then $H^0(Z, T_Z \otimes \Omega_Z(1))\simeq H^0(Z, \sO(1))$ has dimension $N+1$. Thus  we may assume
$\dim Z \geq 2$, namely $c \leq N-3$.

 By Lemma \ref{l.key}, it suffices to show that elements of $\Xi_Z \subset H^0(Z, T_Z \otimes \Omega_Z(1))$ comes from $H^0(Z, T_Y \otimes \Omega_Y(1)|_Z).$ Assume that $\phi \in \Xi_Z \subset \Hom( \wedge^2 V, V)$ is not in $H^0(Z, T_Y \otimes \Omega_Y(1)|_Z)$. Then the collection of homomorphisms $\phi_x: V\otimes \widehat{x} \to V$ at  $x \in Z$ defined by $v \otimes \alpha \mapsto \phi(\alpha, v), \alpha \in \widehat{x}$ induces a nonzero homomorphism $N_{Z/Y} \to N_{Y/\BP V}(1)|_Z$. In other words, a homomorphism of vector bundles on $Z$
$$\sO(m_1) \oplus \cdots \oplus \sO(m_c) \to  \sO( d +1).$$
But by the assumption that $d +1 < m_j$ for all $j$, this must be zero, a contradiction.

\section{Negative examples to Question \ref{q.Question}}\label{s.example}

In this section, we will provide examples of smooth projective
variety with isotrivial VMRT which are not quasi-homogeneous. They
are hyperplane sections of rational homogeneous varieties.

Let $G$ be a simple Lie group of adjoint
type of rank $l$ with Lie algebra $\fg$. We fix a Borel subgroup
$B \subset G$ and the set of roots is denoted by $\Phi$. Let $\{
\alpha_1, \cdots, \alpha_l\}$ be the set of simple roots and
$\Phi^+$ (resp. $\Phi^-$) the set of positive roots (resp.
negative roots). The fundamental weights are denoted by
$\{\lambda_1, \cdots, \lambda_l\}$.  Let $\rho = \frac{1}{2} \sum_{\alpha
\in \Phi^+} \alpha = \sum_{1}^l \lambda_i$ be the half sum of
positive roots.

For any $1 \leq k \leq l$, we denote by $P_k$ the maximal
parabolic subgroup determined by $\lambda_k$. Its Lie algebra
$\fp_k$ is given by
$$
\fp_k = \Phi^{-} \cup \{0\} \cup \Phi^+_k,  \text{where\  }
\Phi^+_k = \{\alpha \in \Phi^+| (\alpha, \lambda_k)=0\}.
$$

Let $L^{\lambda_k}$ be the line bundle defined by the character
$\lambda_k$ on the homogeneous manifold $G/P_k$. Then ${\rm
Pic}(G/P_k) \simeq \mathbb{Z} L^{\lambda_k}$. The line bundle
$L^{\lambda_k}$ is very ample and it induces a natural embedding
$G/P_k \subset \mathbb{P} H^0(G/P_k, L^{\lambda_k}) = \BP
V^{\lambda_k}$, where $V^{\lambda_k}$ is the irreducible
$G$-module with highest weight $\lambda_k$.

The tangent bundle $T_{G/P_k}$ is the homogeneous vector bundle associated to the adjoint representation of $G$ restricted to $P_k$ on the quotient $\fg/\fp_k$, hence the weights of this representation of $P$ are the roots of $G$ that are not in $P$, i.e. roots of $G$ that contain $\alpha_k$ with multiplicity at least 1.  We call $G/P_k$ an {\em irreducible Hermitian symmetric space} (IHSS for short) if the $P_k$-representation $\fg/\fp_k$ is irreducible. They are classified as follows: 

\begin{table}[!hbp]  \label{t}
\begin{tabular}{|c|c|c|c|c|c|c|}
\hline
$\fg$ & $A_l$ & $B_l$ & $C_l$ & $D_l$ & $E_6$  & $E_7$ \\
\hline
$k$ & $1 \leq k \leq l$& 1 & $l$ & $1, l-1, l$ & 1, 6 & 7 \\
\hline
$G/P_k$ & ${\rm Gr}(k, l+1)$ & $\mathbb{Q}^{2l-1}$ & ${\rm Lag}(l, 2l)$ & $\mathbb{Q}^{2l-2}, \mathbb{S}_l$ & $\mathbb{OP}^2$ & $E_7/P_7$\\
\hline
\end{tabular}
\end{table}

In this case, the highest weight of $\fg/\fp_k$ is  the longest root $\beta$.  In the notations of Bourbaki, the longest roots are given by the following:
\begin{table}[!hbp]
\begin{tabular}{|c|c|c|c|c|c|}
\hline
$\fg$ & $A_l$ & $C_l (l \geq 2)$ & $F_4, E_7$ & $B_l (l \geq 3), D_l (l \geq 4), G_2, E_6$ & $E_8$ \\
\hline
$\beta$ & $\lambda_1 + \lambda_l$ & $2 \lambda_1$ & $\lambda_1$ & $\lambda_2$ & $\lambda_8$ \\
\hline
\end{tabular}
\end{table}

\begin{lemma} \label{l.vanishF}
Assume that $G/P_k$ is an IHSS different from projective spaces, then
$H^q(G/P_k, T_{G/P_k} (-1))=0$ for all $q \geq 0$.
\end{lemma}
\begin{proof}
As the maximal weight of $\fg/\fp_k$ is $\beta$, the maximal weight of $T_{G/P_k} (-1)$ is $\mu:=\beta - \lambda_k$.  By our assumption, $\mu$ has coefficient $-1$ at $\lambda_k$, hence  $\mu
+ \rho$ is a singular weight (since $(\alpha_k, \mu+\rho) = 0$).
By Borel-Weil-Bott, we have $H^q(G/P_k, T_{G/P_k} (-1))=0$ for all $q$.
\end{proof}

From now on, we assume that ${\rm Aut}^\circ (G/P_k) = G$. This is
always the case up to replacing $G$ and $P_k$. Let $X \subset
G/P_k$ be a smooth hyperplane section, i.e. $X = G/P_k \cap H$
for a smooth hyperplane $[H] \in \BP (V^{\lambda_k})^{\vee}$.
 Then we have
\begin{lemma} \label{l.stab} Assume that $G/P_k$ is an IHSS different from projective spaces, then

(i) $H^0(X, T_X)$ is identified with the stabilizer of $G$ at $[H]
\in \BP(V^{\lambda_k})^{\vee}$.

(ii) $h^1(X, T_X) = \dim V^{\lambda_k}-1 + h^0(X, T_X) - \dim
\fg$.

(iii) $h^p(X, T_X)=0$ for all $p \geq 2$.
\end{lemma}
\begin{proof}
By the exact sequence
$$
0 \to \sO_{G/P_k} \to \sO_{G/P_k}(1) \to \sO_X(1) \to 0
$$
we get that $h^0(X, \sO_X(1)) = \dim V^{\lambda_k} -1$ and $h^p(X,
\sO_X(1))=0$ for all $p \geq 1$.  Using the exact
sequence
$$
0 \to T_{G/P_k}(-1) \to T_{G/P_k} \to T_{G/P_k}|_X \to 0,
$$
we obtain $$ h^0(X, T_{G/P_k}|_X)= h^0(G/P_k, T_{G/P_k}) =
\fg$$ and $h^p(X, T_{G/P_k}|_X) = 0$ for all $p \geq 1$.  The exact sequence
$$
0 \to T_X \to T_{G/P_k}|_X \to  \sO_X(1) \to 0
$$
gives $$H^0(X, T_X) \subset H^0(X, T_{G/P_k}|_X) = H^0(G/P_k,
T_{G/P_k}),$$  which proves the first claim, while the other two
follow from previous discussions.
\end{proof}

For an irreducible $G$-module $V$, there exists an open subset $U
\subset V$ such that $\dim G\cdot v$ remains the same (and maximal
for all $G$-orbits in $V$). We denote by $m_G(V)$ this dimension.
Then stabilizer of $v \in U$ has dimension $\dim G - m_G(V)$.

\begin{proposition}[\cite{AVE}, Corollary on p.260 ] \label{p.dim}
Assume that $G$ is simple and $V$ an irreducible representation of  $G$.
If $\dim V > \dim G$, then $m_G(V) = \dim G$.
\end{proposition}

On the other hand, we have the following list of $G/P_k$ whose general hyperplane sections are rigid. We will use the notation of the homogeneous varieties in p.466 of \cite{FH}.

\begin{proposition}\label{p.rigid}
Let $G/P_k$ be  an IHSS  and
$X \subset G/P_k$  a general hyperplane section.  Then $X$ is
locally rigid (i.e. $H^1(X, T_X)=0$) if and only if $G/P_k$ is
isomorphic to one of the following:
$$\mathbb{P}^n,  \mathbb{Q}^n, {\rm Gr}(2, n), {\rm Gr}(3, 6), {\rm Gr}(3,
7),  \BS_5, \BS_6, $$ $$ \BS_7, {\rm Lag}(3, 6),
E_6/P_1, E_7/P_7.
$$
\end{proposition}
\begin{proof}
If $\dim V^{\lambda_k} > \dim G$, then $h^0(X, T_X)=0$ by
Proposition \ref{p.dim} and Lemma \ref{l.stab}. Hence $h^1(X,
T_X)=0$ if and only if $\dim V^{\lambda_k} -1 = \dim G$, while
there is no such fundamental representations.

Now assume $\dim V^{\lambda_k} \leq \dim G$.  In \cite{El} (Table 1 on p.46-48),  a
complete list of all irreducible $G$-modules with  $m_G(V) < \dim
G$ together with the stabilizer (denoted by $\mathfrak{h}$) is given.
Then $H^1(X, T_X)=0$ if and only if $\dim V^{\lambda_k} + \dim \mathfrak{h} = \dim \fg +1$.
The Proposition is
obtained by a case-by-case check.
\end{proof}

%

From the above results, we can deduce the following negative examples to Question \ref{q.Question}.

\begin{theorem}\label{t.example}
The following projective manifolds of Picard number 1 have isotrivial VMRT-structures, which is not locally homogeneous. In fact, they do not have continuous automorphism groups. \begin{itemize}
\item[(i)] A general hyperplane section of $\BS_n, n >8$: its  VMRT at a general point is a general hyperplane section of ${\rm Gr}(2, n)$.
\item[(ii)] A general hyperplane sections of ${\rm Lag}(n, 2n), n >4$: its VMRT at a general point is a general hyperplane section of   the second Veronese embedding of $\BP^n$. \end{itemize}
\end{theorem}

\begin{proof} From Lemma \ref{l.stab} and Proposition \ref{p.dim}, we see that the listed varieties do not have continuous automorphism groups. We can see that their VMRT at a general point is as described above from p.466 of \cite{FH}. For (i), the VMRT at a general point is rigid from Proposition \ref{p.rigid}. For (ii), the VMRT at a general point is the second Veronese embedding of a quadric hypersurface. So it is rigid.  Thus the VMRT-structure is isotrivial in both cases. Since the VMRT at a general point is nonsingular and irreducible, Theorem \ref{t.CF} implies that it is not locally homogeneous. \end{proof}

\bigskip
Baohua Fu

Institute of Mathematics, AMSS, Chinese Academy of Sciences,

55 ZhongGuanCun East Road, Beijing, 100190, China
and  School of Mathematical Sciences, University of Chinese Academy of Sciences, Beijing, China

 bhfu@math.ac.cn

\bigskip
Jun-Muk Hwang

 Korea Institute for Advanced Study, Hoegiro 85,

Seoul, 02455, Korea

jmhwang@kias.re.kr

\end{document}